\documentclass{article}

\usepackage{amssymb}
\usepackage{amsmath}
\usepackage{amsfonts}
\usepackage{amsthm}
\usepackage{mathabx}
\usepackage{graphicx}
\usepackage{cite}
\usepackage[breaklinks=true]{hyperref}
\usepackage{graphicx}
\usepackage{tikz}
\usepackage{placeins}
\usepackage{float}
\usepackage{pstricks}
\usepackage{tikz-cd}

\usetikzlibrary{arrows,chains,matrix,positioning,scopes}

\newtheorem{thm}{Theorem}[section]
\newtheorem{theorem}[thm]{Theorem}
\newtheorem{lem}[thm]{Lemma}

\newtheorem{prop}[thm]{Proposition}

\newtheorem{conj}[thm]{Conjecture}
\newtheorem{question}{Question}
\newtheorem{thmm}{Theorem}

\newtheorem{thmk}{Theorem}

\theoremstyle{definition}
\newtheorem{definition}[thm]{Definition}

\newtheorem*{example}{Example}

\theoremstyle{remark}

\newtheorem*{prop proof}{Proof of Proposition}

\newcommand{\RR}{\mathbb{R}}      
\newcommand{\ZZ}{\mathbb{Z}}        

\newcommand{\cH}{\mathcal{H}}
\newcommand{\hyp}{\mathfrak{h}}

\newcommand{\minset}{{\rm{Min}}}

\numberwithin{equation}{section}

\title{Groups with arbitrary cubical dimension gap}

\author{Robert Kropholler and Chris O'Donnell}

\begin{document}
\maketitle

\begin{abstract}
	We prove that if $G = G_1\times\dots\times G_n$ acts essentially, properly and cocompactly on a CAT(0) cube complex $X$, then the cube complex splits as a product. We use this theorem to give various examples of groups for which the minimal dimension of a cube complex the group acts on is strictly larger than that of the minimal dimension of a CAT(0) space upon which the group acts. 
\end{abstract}

\section{Introduction}

The purpose of this paper is to study splitting theorems for CAT(0) cube complexes with a proper, cocompact, essential action by a product of two groups. We use this to give a sequence of groups with arbitrary dimension gaps. 

There are many splitting results for actions of products on CAT(0) spaces. See for instance, \cite[p.239 Theorem 6.21]{bridsonHaefliger} and the preceeding discussion as well as \cite{monod} for a more recent work. We prove a version for CAT(0) cube complexes.

We say that the $G$ has {\em the Abelian intersection property (AIP)} if there is a sequence of highest Abelian subgroups $A_1, \dots, A_n<G$ such that $\bigcap A_i = \{e\}$.

\begin{thmm}
	Let $G_1, \dots, G_n$ be finitely generated groups with AIP. Suppose that $G = G_1\times\dots\times G_n$ acts on a CAT(0) cube complex $X$ properly, cocompactly and essentially. Then $X = X_1\times \dots\times X_n$ and $G_i$ acts properly cocompactly on $X_i$.  
\end{thmm}

By \cite{caprace_cubical} that we can always restrict to a subcomplex where the action is essential, thus this assumption should not be troubling. 

We use the AIP property heavily throughout and conjecture that any CAT(0) group with finite center in all finite index subgroups has AIP. 

We use this cubical splitting theorem to obtain groups with arbitrary gaps between the CAT(0) dimension and cubical dimension. Examples of this kind have previously been given by \cite{jankiewicz} using small cancellation theory. We build on the work of \cite{bridsondimension} to give examples with explicit complexes realising both the CAT(0) dimension and the cubical dimension. Namely, we prove the following:

\begin{thmm}
	For each $n\geq 1$. There exists a group of CAT(0) dimension $2n$ but cubical dimension $3n$. 
\end{thmm}

We also use this theorem to construct manifold examples. In this case the manifolds acting properly on a product of copies of $\mathbb{H}^3$. This realises their geometric dimension. While we can bound the cubical dimension from below, in these examples we do not have explicit cubulations. However, the groups are cubulated by \cite{bergeronwise}. 

\begin{thmm}
	There are closed aspherical manifolds whose fundamental groups have arbitrary dimension gaps. 
\end{thmm}

Since we are using the product decomposition theorem to gain arbitrary gaps these manifolds will have arbitrarily large dimension. Using the results of \cite{jankiewicz} and the reflection trick of \cite{davis} we can construct 4-dimensional aspherical manifolds with arbitrary dimension gap. 

\begin{thmm}
	For each $n$, there is a group $G_n$ which is the fundamental group of an aspherical 4-manifold that does not act properly on any cube complex of dimension less than $n$. 
\end{thmm}

It is known that surface groups can be cubulated in dimension 2. This leaves the question in what happens in dimension 3. In particular, Genevieve Walsh asks the following question. 

\begin{question}
	Is there a family of hyperbolic 3-manifolds $M_n$ such that $\pi_1(M_n)$ does not act geometrically on a CAT(0) cube complex of dimension $<n$? 
\end{question}

We thank Genevieve Walsh for helpful comments. This project contains work included in the second author's PhD thesis undertaken at Tufts university.

\section{Preliminaries}

\begin{definition}
	An Abelian subgroup $A \leq G$ is {\em highest} if $A$ does not have a finite index subgroup that lies in an Abelian subgroup of higher rank.
\end{definition}

\begin{definition}
	A group $G$ has {\em the Abelian intersection property (AIP)} if there is a sequence of highest Abelian subgroups $A_1, \dots, A_n<G$ such that $\bigcap A_i = \{e\}$. 
\end{definition}

\begin{lem}
	Let $G$ be a CAT(0) group with AIP. Then $G$ has finite center. 
\end{lem}
\begin{proof}
	Suppose that the center is infinite. Since the group $G$ is CAT(0), the center has an element $g$ of infinite order. Every highest abelian subgroup $A$ contains $g^n$, where $n$ depends on $A$. The intersection of finitely many highest abelian subgroups will thus contain some power of $g$. 
\end{proof}

\begin{lem}
	If $G$ is a non-elementary hyperbolic group, then $G$ has AIP. 
\end{lem}
\begin{proof}
	Let $g_1, g_2$ be two elements that generate a free group in $G$. Then $\langle g_i\rangle$ are highest abelian subgroups and they intersect trivially. 
\end{proof}

In what follows we will prove a product decomposition theorem for groups with AIP. As such it is of interest to know when CAT(0) groups have AIP. We pose the following conjecture. 

\begin{conj}
	Suppose that $G$ is a CAT(0) group. Suppose that $G$ and every finite index subgroup have finite center. Then $G$ has AIP. 
\end{conj}

While it is clear that the center of $G$ must be finite, we also require the condition on finite index subgroups as there are Bieberbach groups which have trivial center. Such groups are virtually $\ZZ^n$ and so certainly do not have AIP.

\begin{lem}\label{lemma:quasilines}
	Let $G = G_1\times\dots\times G_n$ be a group acting properly, cocompactly and essentially on a CAT(0) cube complex $X$. Suppose that $G$ has AIP. Let $A$ be a highest Abelian subgroup of $G$. Then $A = A_1\times\dots\times A_n$, with $A_i\subset G_i$. Moreover, each $A_i$ contains a finite index subgroup which acts on a product of quasi-lines in $X$. 
\end{lem}
\begin{proof}
	Let $B_i = p_i(A)$ where $p_i$ is the projection of $G$ to $G_i$. Since $B_i$ is Abelian and the projection map splits we see that $B_i<A$. Suppose that $B_i$ is not highest in $G_i$, then there is a subgroup $C_i<G_i$ such that $B_i\cap C_i$ has finite index in $B_i$. We can now take $H_i = p_i^{-1}(B_i\cap C_i)\cap A$. This has finite index in $A$ and every element of $C_i$ commutes with every element $H_i$. Thus we have that $A$ is a product of highest abelian subgroups. 
	
	By the cubical flat torus theorem \cite{wise_cubical} we see that $A$ acts on a product of quasi-lines $C_1\times\dots\times C_m$. We will show that $A_1$ has a finite index subgroup which acts on a subproduct $C_1\times\dots\times C_k$ and acts trivially on $C_{k+1}\times\dots\times C_m$. 
	
	By \cite[Lemma 4]{wise_cubical}, there is a finite index subgroup of $A$ with a preferred set of generators $S = \{a_1, \dots a_m\}$ such that $a_i$ acts non-trivially only on $C_i$ upon which it acts by translation. 
	
	Let $A_i^j$ be a sequence of subgroups witnessing AIP for the group $G_i$. We will assume that $A_i^0 = A_i$. 
	
	By \cite[Theorem 7]{wise_cubical}, the intersection $$A\bigcap \left(A_1\times \prod_{i=2}^{n}A_i^{j_i}\right)$$ is commensurable with a subgroup generated by a subset of $S$. In particular each has a finite index subgroup generated by powers of the $\alpha_i$. Let $B$ be one such intersection and $B'$ be another. Let $C$ and $C'$ be the finite index subgroups of $B$ and $B'$ respectively. Then $C\cap C'$ is finite index in $B\cap B'$. After taking a finite number of intersections we will arrive at a finite index subgroup of $A_1$ generated by powers of elements of $S$. This subgroup will act on the subproduct of quasi-lines as stated. The proof is similar for the other $A_i$. 
\end{proof}

It is useful to be able to pass back from products of groups to the groups themselves. It is useful to know that AIP is preserved under this transition. 

\begin{prop}
	Let $G = G_1\times\dots\times G_n$. Then $G$ has AIP if and only if $G_i$ has AIP for each $i$. 
\end{prop}
\begin{proof}
	If each of the $G_i$ has AIP, then taking the products of the highest abelian groups in each factor gives the desired family. 
	
	If $G$ has AIP, then we can find a sequence of highest abelian subgroups intersecting trivially. The projection to each factor gives a highest abelian subgroup of the factor. These must intersect trivially in each factor. 
\end{proof}

\section{A product decomposition theorem for cube complexes}

In this section we wish to prove a product decomposition for cubical groups with AIP. We begin by getting a splitting result for the cube complex. 

\begin{prop}
	Let $G_1, \dots, G_n$ be finitely generated groups with AIP. Suppose that $G = G_1\times\dots \times G_n$ acts properly, cocompactly and essentially on a CAT(0) cube complex $X$. Then $X$ decomposes as a product of CAT(0) cube complexes $X = X_1\times\dots\times X_n$. 
\end{prop}
\begin{proof}
	We will show that the set of hyperplanes $\cH(X)$ splits as a disjoint union $\cH(X) = \cH_1\sqcup\dots\sqcup \cH_n$ and that each hyperplane in $\cH_i$ intersects each hyperplane in $\cH_j$. 
	
	By \cite[Proposition 3.12]{caprace_cubical}, each hyperplane $\hyp$ is skewered by some element of $G$. Suppose that this element is $(g_1, \dots, g_n)$. Let $A_i<G_i$ be a highest abelian subgroup containing a power of $g_i$. Let $A = A_1\times\dots\times A_n$. Then $A$ is a highest abelian subgroup of $G$. Thus we get an action on a product of quasi-lines in which there is a finite index subgroup $\bar{A}_i$ of each $A_i$ acting on a subproduct by Lemma \ref{lemma:quasilines}. 
	
	By the flat torus theorem \cite[p.244 Theorem 7.1]{bridsonHaefliger}, the min set of $A$ decomposes as $E\times Z$ where $E = \RR^n$ and $A$ acts trivially on $Z$. The product of quasi-lines $Y$ constructed in \cite{wise_cubical} is the dual cube complex to the set of hyperplanes intersecting $E$. Since $\hyp$ is skewered by $(g_1, \dots, g_n)$ we see that $\hyp$ intersects $E$. So $\hyp$ is a hyperplane of $Y$. 
	
	The hyperplane $\hyp$ is dual to some quasi-line $C_l$ of $Y$. Thus it is skewered by the elements that translate along $C_l$. By Lemma \ref{lemma:quasilines} this can only consist of elements in $\bar{A}_i$ for one $i$. Thus the hyperplane is skewered by $g_i$ for some $i$ and every $g_j$ for $j\neq i$ does not skewer $\hyp$. 
	
	Now suppose that $\hyp$ is skewered by another element $(g_1', \dots, g_n')$ we will show that it is also skewered by $g_i'$. We know that some component $g_j'$ of $(g_1', \dots, g_n')$ skewers $\hyp$. 
	
	For the sake of a contradiction, we assume that $j\neq i$. We can construct a highest abelian subgroup $B$ containing powers of $g_i$ and $g_j'$. By our previous argument we see that $B$ acts on a product of quasi lines. We also see that there are $k_i$ and $k_j$ such that $g_i^{k_i}$ and $g_j^{k_j}$ only act on quasi-lines corresponding to the factors $i, j$ respectively. Thus we see that $j = i$. 
	
	We can now divide $\cH(X)$ into $n$ sets $\cH_1\sqcup\dots\sqcup \cH_n$ where $\cH_i$ is the set of hyperplanes skewered by elements in $G_i$. 
	
	We now show that any hyperplane $\hyp_i\in \cH_i$ intersects any hyperplane $\hyp_j\in \cH_j$ for $j\neq i$. 
	
	Let $\hyp_i\in\cH_i$ and $\hyp_j\in\cH_j$ be hyperplanes. Let $g_i\in G_i$ be an element which skewers $\hyp_i$ and $g_j\in G_j$ be an element which skewers $\hyp_j$. By replacing $g_i$ and $g_j$ with appropriate powers we can find a highest abelian subgroup $A<G$ containing $g_i$ and $g_j$. Let $p$ be the rank of $A$. By the flat torus theorem \cite[p.244 Theorem 7.1]{bridsonHaefliger} there is a copy of $\RR^p$ contained in $X$ upon which $A$ acts. The element $g_j$ acts by translation and does not skewer the hyperplane $\hyp_i$. Let $x\in \hyp_i$ we can see that $\langle g_j\rangle\cdot x$ is contained in $\hyp_i$. However since $g_j$ does skewer $\hyp_j$ we can see that $\langle g_j\rangle\cdot x$ contains points on both sides of $\hyp_j$ and thus $\hyp_i$ must intersect $\hyp_j$ by convexity. 
	
	By \cite{caprace_cubical}, we see that $X$ splits as a product $X = X_1\times\dots\times X_n$ where $X_i$ is the cube complex dual to $\cH_i$. 
\end{proof}

Since each $\hyp\in\cH_i$ is only skewered by elements in $G_i$ we see that the action decomposes as a product and that each $g_i$ acts elliptically on $X_1\times\dots\times X_{i-1}\times X_{i+1}\times\dots\times X_n$. 

\begin{theorem}
	Let $G$ and $X$ be as above. Then $G_i$ acts properly and cocompactly on $X_i$. 
\end{theorem}
\begin{proof}
	We first show that $G_i$ acts cocompactly on $X_i$. There is a splitting of $X_i$ as a product $X_{i1}\times\dots\times X_{im}$ and a finite index subgroup $\bar{G}<G$ which does not permute the factors of this splitting, we can assume that $\bar{G} = \prod_{k=1}^n \bar{G}_k$ where $\bar{G}_k$ is finite index in $G_k$. 
	
	Fix $j$ and consider $X_{ij}$. By construction there is an action of $\bar{G}$ on $X_{ij}$. Since $X_{ij}$ is irreducible, \cite[Proposition 5.1]{caprace_cubical} shows it contains a pair $\hyp, \hyp'$ of strongly separated hyperplanes. By the double skewering lemma there is some $g = (g_1, \dots, g_n)\in \bar{G}$ such that $g\hyp\subsetneq\hyp'\subsetneq\hyp$. 
	
	Since $\hyp$ is a hyperplane of $X_{ij}$ we see that $\hyp, \hyp'$ are skewered by $g_i$. So by \cite{caprace_cubical} every axis of $g_i$ intersects both $\hyp$ and $\hyp'$ and are all in a bounded neighbourhood of each other. 
	
	Let $g_l'\in \bar{G_l}$ for some $l\neq i$. Then $g_i$ and $g_l'$ commute. Thus, $g_l'$ preserves the axes of $g_i$ and we get an action of $g_l'$ on $\minset(g_i) = Y\times \RR$ where $Y$ is bounded. Since $g_l'$ does not skewer $\hyp$ we see that it must fix a point of the $\RR$ factor and thus acts by reflection or the identity. However since $g_i$ acts by translation and $g_i$ and $g_l'$ commute we see that $g_l'$ must act by the identity. Thus $g_l'$ acts on $Y$ and has a bounded orbit in $X_i$. 
	
	So the subgroup $\prod_{k\neq i}\bar{G}_k$ has a bounded orbit in $X_i$. Thus we see that $\hat{G}_i = \prod_{k\neq i}G_k$ also has a bounded orbit in $X_i$. However $G_i\times \hat{G}_i$ acts cocompactly and so $G_i$ must have a cocompact action. 
	
	We now prove that the action is proper. Suppose for a contradiction that the action is not proper. Then there is a bounded set $A\subset X_i$ such that $\{g\in G_i\mid g_iA\bigcap A\neq \emptyset\}$ is infinite. Since $G_i$ acts on $X_j$ with bounded orbits for $j\neq i$, we can take a bounded set $C_j$ containing an orbit. Then $\{g\in G_i \mid g(A\times \prod_{k\neq i}C_k)\neq\emptyset\}$ is infinite contradicting the properness of the action of $G$ on $X$. 
\end{proof}

Thus we have now proved the following:

\begin{thmk}\label{thm:productdecomp}
	Let $G_1, \dots, G_n$ be finitely generated groups. Suppose that $G = G_1\times\dots\times G_n$ acts on a CAT(0) cube complex $X$ properly, cocompactly and essentially. Then $X = X_1\times \dots\times X_n$ and $G_i$ acts properly cocompactly on $X_i$.  
\end{thmk}

The induced action of $G_i$ on $X_j$ may not be trivial for $j\neq i$. We do have the following:

\begin{prop}
	For every $i$ there is a finite index subgroup $G_i'<G_i$ such that $G_i'$ acts trivially on $X_j$ for $j\neq i$. 
\end{prop}
\begin{proof}
	There is an action of $G_i$ on $X_j/G_j$. Since $G_j$ acts cocompactly on $X_j$, this is an action on a compact space and so there is a finite index subgroup $G_{ij}<G_i$ which acts trivially on $X_j/G_j$. This gives a homomorphism $G_{ij}\to G_j$ and since every element of $G_{ij}$ commutes with every element of $G_j$, the image is contained in the center which is finite. And so taking the kernel gives us the desired finite index subgroup. We can now intersect these kernels for each $j$ to obtain $G_i'$ which acts trivially on $X_j$ for $j\neq i$. 
\end{proof}

Ideally we would like to say that the group $G_i$ acts trivially on the factors $X_j$ for $j\neq i$ although this is not possible as the following example illustrates. 

\begin{example}

Let $F_2$ be the free group on two generators. Let $T$ be the 4-regular tree and let $T_{\square}$ be the universal cover of the cube complex consisting of two squares and identifying two diagonal vertices on each square. We can see $T_{\square}$ as being obtained from the tree by replacing each edge with the diagonal of a square. There is a natural action of $F_2$ on $T_{\square}$ which is proper, cocompact and essential. 

Let $F_2\times F_2$ act on $T_{\square}\times T$ where the first $F_2$ acts trivially on $T$ and by covering transformations on $T_{\square}$. Let the second factor act on $T_4$ in the usual way and on $T_{\square}$ by letting each generator reflect in the diagonal of the square. This is a proper, cocompact and essential action although it is clear that the action of the second factor on $T_{\square}$ is not trivial. 

\end{example}

\section{A cubical dimension gap}

In \cite{bridsondimension}, a group $G$ is constructed where the CAT(0) dimension of finite index subgroups $H$ is strictly smaller than the CAT(0) dimension of $G$. We use an explicit group with CAT(0) dimension 2 but cubical dimension 3. We then use the product decomposition theorem discussed earlier to show that this example can be used to get arbitrary gaps between CAT(0) dimension and cubical dimension. 

We require the following proposition from \cite{bridsondimension}.

\begin{prop} \label{prop:commutatortree}
	Let $F_2$ act properly on an $\RR$-tree, then $|[a, b]|\neq |a|$ or $|[a, b]|\neq |b|$. 
\end{prop}

Let $G = ((F_2\times F_2)\ast_{\ZZ})\ast_{\ZZ}$ given by the presentation $$ \langle x, y, a, b, s, t\mid [x, a], [x, b], [y, a], [y, b], sas^{-1} = [a, b] = tbt^{-1}\rangle.$$

We will prove that this group has a dimension gap and satisfies property (AIP). Thus taking products we can realise arbitrarily large dimension gaps between CAT(0) and cubical dimension. 

\begin{prop}\label{dimension2}
	The group $G$ has CAT(0) dimension 2. 
\end{prop}
\begin{proof}
	A CAT(0) complex can be built as in \cite{bridsondimension}. We take 4 tori which are built from rhombi with one side of length 4 and vertical height 1. We can then glue on two cylinders to realise the relations $sas^{-1} = [a, b] = tbt^{-1}$. This is depicted in Figure \ref{fig:cat0structure} 
\end{proof}
\begin{figure}[h!]
	\centering
\begin{tikzpicture}[scale=0.58]
\draw [thick,|-|, dashed](-2, 0) -- (2, 2);
\node [above left] at (0, 1) {$4$};
\draw [thick,->](0,0) -- (2,1)node[above]{$a$};
\draw [thick](2,1) -- (4,2);
\draw [thick,->](0,0) -- (2.5,0)node[above]{$x$};
\draw [thick](1,0) -- (5,0);
\draw [thick,->](4,2) -- (6.5,2)node[above]{$x$};
\draw [thick](6.5,2) -- (9,2);
\draw [thick,->](5,0) -- (7,1)node[above]{$a$};
\draw [thick](7,1) -- (9,2);
\begin{scope}[shift={(8,0)}]
\draw [thick,->](0,0) -- (2,1)node[above]{$b$};
\draw [thick](2,1) -- (4,2);
\draw [thick,->](0,0) -- (2.5,0)node[above]{$x$};
\draw [thick](1,0) -- (5,0);
\draw [thick,->](4,2) -- (6.5,2)node[above]{$x$};
\draw [thick](6.5,2) -- (9,2);
\draw [thick,->](5,0) -- (7,1)node[above]{$b$};
\draw [thick](7,1) -- (9,2);
\draw [thick,|-|, dashed](10, 0) -- (10, 2);
\node [left] at (10, 1) {$1$};
\end{scope}
\begin{scope}[shift={(0, -4)}]
\draw [thick,->](0,0) -- (2,1)node[above]{$a$};
\draw [thick](2,1) -- (4,2);
\draw [thick,->](0,0) -- (2.5,0)node[above]{$y$};
\draw [thick](1,0) -- (5,0);
\draw [thick,->](4,2) -- (6.5,2)node[above]{$y$};
\draw [thick](6.5,2) -- (9,2);
\draw [thick,->](5,0) -- (7,1)node[above]{$a$};
\draw [thick](7,1) -- (9,2);
\end{scope}
\begin{scope}[shift={(8,-4)}]
\draw [thick,->](0,0) -- (2,1)node[above]{$b$};
\draw [thick](2,1) -- (4,2);
\draw [thick,->](0,0) -- (2.5,0)node[above]{$y$};
\draw [thick](1,0) -- (5,0);
\draw [thick,->](4,2) -- (6.5,2)node[above]{$y$};
\draw [thick](6.5,2) -- (9,2);
\draw [thick,->](5,0) -- (7,1)node[above]{$b$};
\draw [thick](7,1) -- (9,2);
\end{scope}
\begin{scope}[shift={(2,-9)}]
\draw [thick,->](0,0) -- (0,1.5)node[left]{$s$};
\draw [thick](0,1) -- (0,3);
\draw [thick,->](0,0) -- (2,0)node[above]{$[a, b]$};
\draw [thick](1,0) -- (4,0);
\draw [thick,->](0,3) -- (2, 3)node[above]{$a$};
\draw [thick](2,3) -- (4,3);
\draw [thick,->](4,0) -- (4,1.5)node[right]{$s$};
\draw [thick](4,1) -- (4,3);
\end{scope}
\begin{scope}[shift={(10,-9)}]
\draw [thick,->](0,0) -- (0,1.5)node[left]{$t$};
\draw [thick](0,1) -- (0,3);
\draw [thick,->](0,0) -- (2,0)node[above]{$[a, b]$};
\draw [thick](1,0) -- (4,0);
\draw [thick,->](0,3) -- (2, 3)node[above]{$b$};
\draw [thick](2,3) -- (4,3);
\draw [thick,->](4,0) -- (4,1.5)node[right]{$t$};
\draw [thick](4,1) -- (4,3);
\end{scope}
\end{tikzpicture}
\caption{The 2-dimensionsal CAT(0) structure for $G$}
\label{fig:cat0structure}
\end{figure}
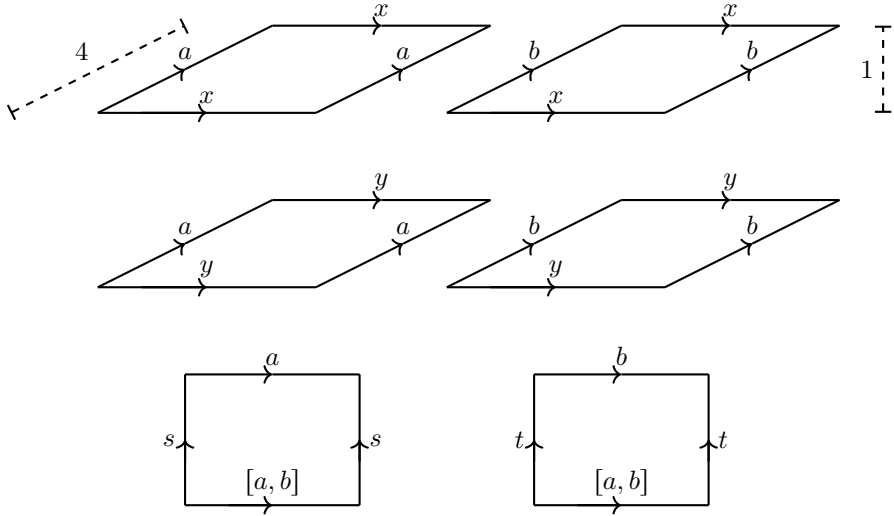

\begin{prop}\label{cubicaldim3}
	The group $G$ has cubical dimension 3. 
\end{prop}
\begin{proof}
	We can construct a 3-dimensional cube complex by taking a torus with one boundary component cubulated below. Taking the product with a figure 8 graph and gluing on the two cylinders. This is depicted in Figure \ref{fig:cubecomplex}. 
	
	We must now show that it cannot have cubical dimension 2. 
	
	Assume that $G$ acts properly cocompactly on a 2-dimensional CAT(0) cube complex $X$. By passing to a subcomplex we can assume that the action is essential.
	
	Consider the group $\langle a, x\rangle$, this is a highest Abelian subgroup. Thus there is a product of quasi-lines which this group stabilises. Since $X$ is 2-dimensional, this product is actually a product of 2 lines. Similarly the groups $\langle a, y\rangle, \langle b, x\rangle$ and $\langle b, y\rangle$ have this property. 
	
	We can get a basis for a finite index subgroup which acts as the standard product action on these lines. Using \cite{wise_cubical} as in the proof of Lemma \ref{lemma:quasilines} we see that the elements $a, b, x, y$ all have cubical axes. 
	
	Consider the min set of $x$. Since we are in a 2-dimensional cube complex this min set is of the form $T\times \RR$ for some tree $T$. The centraliser of $t$ stabilises this min set. An application of normal forms show that this centraliser is $\langle a, b, t\rangle$. The elements $a$ and $b$ translate in a direction orthogonal to the $\RR$ direction. This means that $\langle a, b\rangle$ stabilises $T\times \{0\}$ in this splitting. This action also realises the translation lengths of $a$, $b$ and $[a, b]$. However, this contradicts Proposition \ref{prop:commutatortree}. We conclude that this group must have cubical dimension $>2$. 
\end{proof}
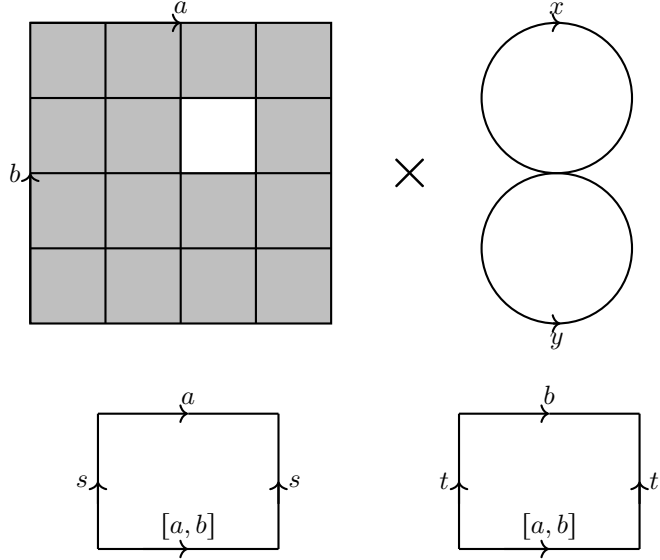
\begin{figure}
	\centering
	\begin{tikzpicture}
		\draw [fill=lightgray, thick] (0,0) rectangle (4,4);
		\draw [thick] (1, 0) -- (1, 4);
		\draw [thick] (2, 0) -- (2, 4);
		\draw [thick] (3, 0) -- (3, 4);
		\draw [thick] (0, 1) -- (4, 1);
		\draw [thick] (0, 2) -- (4, 2);
		\draw [thick] (0, 3) -- (4, 3);
		\draw [thick, ->] (0, 4) -- (2, 4) node[above]{$a$};
		\draw [thick, ->] (0, 0) -- (0, 2) node[left]{$b$};
		\draw [fill=white, thick] (2,2) rectangle (3,3);
		\node [left] at (5.5, 2) {\Huge$\times$};
		\draw [thick] (7, 3) circle [radius=1];
		\draw [thick] (7, 1) circle [radius=1];
		\node [above] at (7, 4) {$x$};
		\node [below] at (7, 0) {$y$};
		\draw [thick, ->] (6.95, 0) -- (7.05, 0);
		\draw [thick, ->] (6.95, 4) -- (7.05, 4);
		
		\begin{scope}[scale=0.6,shift={(1.5,-5)}]
		\draw [thick,->](0,0) -- (0,1.5)node[left]{$s$};
		\draw [thick](0,1) -- (0,3);
		\draw [thick,->](0,0) -- (2,0)node[above]{$[a, b]$};
		\draw [thick](1,0) -- (4,0);
		\draw [thick,->](0,3) -- (2, 3)node[above]{$a$};
		\draw [thick](2,3) -- (4,3);
		\draw [thick,->](4,0) -- (4,1.5)node[right]{$s$};
		\draw [thick](4,1) -- (4,3);
		\end{scope}
		\begin{scope}[scale=0.6,shift={(9.5,-5)}]
		\draw [thick,->](0,0) -- (0,1.5)node[left]{$t$};
		\draw [thick](0,1) -- (0,3);
		\draw [thick,->](0,0) -- (2,0)node[above]{$[a, b]$};
		\draw [thick](1,0) -- (4,0);
		\draw [thick,->](0,3) -- (2, 3)node[above]{$b$};
		\draw [thick](2,3) -- (4,3);
		\draw [thick,->](4,0) -- (4,1.5)node[right]{$t$};
		\draw [thick](4,1) -- (4,3);
		\end{scope}
		
	\end{tikzpicture}
	\caption{A 3-dimensional CAT(0) cubical structure for $G$. }
	\label{fig:cubecomplex}
\end{figure}

\begin{thmk}
	For each $n\geq 1$. There exists a group of CAT(0) dimension $2n$ but cubical dimension $3n$. 
\end{thmk}
\begin{proof}
	We will show that the group $G^n$ has CAT(0) dimension $2n$, but cubical dimension $3n$. 
	
	The group $G^n$ has cohomological dimension $2n$, thus it's CAT(0) dimension is at least $2n$. Taking a product of the spaces from Proposition \ref{dimension2} realises the CAT(0) dimension. 
	
	Let $X$ be any cube complex upon which $G^n$ acts properly cocompactly. By passing to a subcomplex we can assume that the action is essential. Thus, using Theorem \ref{thm:productdecomp} we get a splitting of $X$ as a product $X_1\times \dots \times X_n$. Where each factor inherits a $G$ action. By Proposition \ref{cubicaldim3}, we see that each factor must have dimension at least 3 and $X$ has dimension at least $3n$. We can realise the bound of $3n$ by taking the cube complex constructed in Proposition \ref{cubicaldim3}. 
\end{proof}

\section{Manifold examples}

In this section we will find aspherical manifolds with arbitrary gaps between their CAT(0) dimension and cubical dimension. We begin by finding examples using the product decomposition theorem. To start with we build on \cite{taoli}, in which there are examples of hyperbolic 3-manifolds which are not homeomorphic to any 3-dimensional CAT(0) cube complex. We give examples of 3-manifold groups which cannot act geometrically on any 3-dimensional cube complex. 

\begin{thm}\label{prop:3manifoldfinitegap}
	There exists hyperbolic 3-manifold groups with a finite gap between their CAT(0) dimension and their cubical dimension. 
\end{thm}
\begin{proof}
	Let $\bar{M}$ be an orientable and irreducible 3-manifold whose boundary is an incompressible torus that does not contain any closed, nonperipheral, embedded, incompressible surfaces.
	For example, $\bar{M}$ could be the figure-8 knot complement. 
	Let $M$ be a hyperbolic 3-manifold obtained by a Dehn filling of $\bar{M}$ that has the 4-plane property.
	This manifold exists by \cite{taoli}, when the filling is long enough. 

	Let $G = \pi_1(M)$.
	Then $G$ acts properly, cocompactly, and by isometries on $\mathbb{H}^3$.
	This shows that $G$ has CAT(0) dimension 3. 

	We will now show that such a manifold has cubical dimension $\geq4$. 
	Let $X$ be a CAT(0) cube complex with a proper cocompact and essential $G$ action. 
	Let $\mathcal{H}$ be a hyperplane with stabiliser $H$
	This is a quasiconvex subgroup of $G$. 
	Consider $\partial \mathcal{H} = \partial H\subset \partial G = S^2$. 
	Since the action is essential we know that $\partial H\neq S^2$. 
	Thus we can consider the domain of discontinuity $\Omega$ for $H\curvearrowright S^2$.
	Let $D$ be a disc in $\Omega$ and $N\leq H$ be the stabiliser of $\partial D = S^1$.
	Then $N$ is a quasiconvex surface subgroup of $H$ \cite{ahlfors}. 
	We now obtain a copy $C$ of $\mathbb{H}^2\subset \mathbb{H}^3$ stabilised by $N$. 
	
	We can translate this copy of $C$ by the group action. 
	We now that at least 4 copies of $C$ will intersect transversely \cite{taoli}.
	In $S^2$ we see that three translates $D_1, D_2, D_3$ of $\partial D$ intersect $D$ non-trivially. 
	Consider the corresponding three translates $P_1, P_2, P_3$ of $P = \partial H$ in $S^2$. 
	Since $D$ is a disc in $\Omega$ and all the the translates $P_i$ intersect $D$ non-trivially, we see that $P_i \neq P$ for all $i$. 
	Similarly, we see that $P_i\neq P_j$. 
	Thus we have 4 translates of $P_i$ that intersect each other. 
	This corresponds to 4 translates of $H$ that intersect each other in $X$. 
	Since hyperplanes satisfy the Helly property we see that they all intersect.
	Thus we can find a cube in which these hyperplanes intersect.
	The 4 hyperplanes are dual to edges of this cube and thus $X$ has dimension at least 4. 
\end{proof}

\begin{thmk}
	There are manifolds $M_n$ such that the cubical dimension of $\pi_1(M_n)$ is at least $4n$ and the CAT(0) dimension is $3n$. 
\end{thmk}
\begin{proof}
	Let $M$ be the manifold constructed in Proposition \ref{prop:3manifoldfinitegap}. Let $M_n = M\times\dots\times M$ be a product of $n$ copies of $M$. 
	
	We can see that the CAT(0) dimension of $M_n$ is $3n$ since it acts properly cocompactly on $(\mathbb{H}^3)^n$. 
	
	To see that the cubical dimension is at least $4n$, assume that $\pi_1(M_n)$ acts properly cocompactly on a cube complex $X$. Theorem \ref{thm:productdecomp} implies that there is a subcomplex of $X$ that splits as a product $X_1\times\dots \times X_n$ and $\pi_1(M)$ acts properly cocompactly on $X_i$, thus $X_i$ must have dimension at least $4$ by Proposition \ref{prop:3manifoldfinitegap}. Thus we conclude that $X$ had cubical dimension at least $4n$. 
\end{proof}

We finish this section with an observation relating to the examples constructed in \cite{jankiewicz}. Namely, the examples can be used to construct aspherical 4-manifolds with an arbitrary dimension gap. 

\begin{thmk}
	There exists a family $M_n$ of closed aspherical 4-manifolds such that $G_n = \pi_1(M_n)$ does not act on a cube complex of dimension less than $n$. 
\end{thmk}
\begin{proof}
	Let $H_n$ be the group constructed in \cite{jankiewicz} which does not act properly on any CAT(0) cube complex of dimension $<n$. The group $H_n$ is a small cancellation group so the presentation 2-complex is a classifying space. We can embed this classifying space into $\RR^4$ \cite{stallings} and take a neighbourhood to get a 4-manifold which is a classifying space. 
	
	Triangulating the boundary $N_n$ of this 4-manifold we can apply the Davis trick \cite{davis} to obtain a closed aspherical 4-manifold $M_n$. The fundamental group $\Gamma_n$ of $M_n$ retracts onto $H_n$. Thus $\Gamma_n$ cannot act properly on any cube complex of dimension less than $n$ as we would then get a proper action of $H_n$. 
\end{proof}

The group from the previous theorem acts properly on a cube complex. Let $X_n$ be a cube complex upon which $H_n$ acts properly and let $W_{n}$ be the Davis complex for the universal cover of $N_n$. Then the group $\Gamma_n$ acts properly on $X_n\times W_n$. However, this action is not cocompact. We hope that in the future we can promote these examples to cocompactly cubulated groups.

\bibliographystyle{plain}
\bibliography{bib}

\end{document}